\documentclass[12pt]{amsart}
\usepackage{amssymb}
\usepackage{amsfonts}
\usepackage{graphicx}
\usepackage{epsfig}
\usepackage{amsmath}

\theoremstyle{plain}
\newtheorem{theorem}{Theorem}[section]
\newtheorem{proposition}[theorem]{Proposition}

\theoremstyle{definition}

\newcommand{\R}{\mathbb{R}}
\newcommand{\C}{\mathbb{C}}
\newcommand{\bc}{\mathbb{C}}

\newcommand{\bR}{{\mathbb R}}

\def\new#1{\textsl{\textbf{{#1}}}}

\def\H{\mathbb{H}}

\def\<{\langle}
\def\>{\rangle}

\def\hh#1{{{\bf H}^{#1}_{\H}}}
\def\P{\mathbb P}
\begin{document}
\title
        {Quaternionic hyperbolic Fuchsian groups}
        \author{Joonhyung Kim}

\address{Joonhyung Kim \\
    Department of Mathematics, Konkuk University\\
           1 Hwayang-dong, Gwangjin-gu\\
           Seoul 143-701, Republic of Korea}
\email{calvary\char`\@snu.ac.kr}%
        \date{}
        \maketitle

\begin{abstract}
In this paper we give the characterization of Fuchsian groups acting on quaternionic hyperbolic $2$-space.
\end{abstract}
\footnotetext[3]{\sl{The author was supported by NRF
grant 2011-0001565.}}

\section{Introduction}
The following theorem is found by B.Maskit in \cite{Ma}(p108).
\begin{theorem}
Let $G \subset \textbf{SL}(2,\bc)$ be a non-elementary Kleinian group in which $tr^2(g) \geq 0$ for all $g \in G$. Then $G$ is Fuchsian.
\end{theorem}
Here, we say that a Kleinian group is elementary if its limit set consists of at most two points, and the others are called non-elementary. One well-known fact is that a non-elementary Kleinian group contains at least two loxodromic elements with distinct fixed points. Furthermore, a Kleinian group is called Fuchsian if it keeps invariant a disk in Riemann sphere. This theorem gives a sufficient condition for a Kleinian group to be Fuchsian group in the group of holomorphic isometries in real hyperbolic plane. After that, B.Xie proved similar result in complex hyperbolic case in \cite{Xie}. In complex hyperbolic space, there is no totally geodesic submanifold of codimension one, but there are two kinds of totally geodesic submanifolds of codimension two. Hence there are two kinds of Fuchsian groups, called $\mathbb{C}$-Fuchsian and $\bR$-Fuchsian. Therefore we say a Kleinian group $G$ in the group of holomorphic isometries of complex hyperbolic space is Fuchsian if $G$ is either $\mathbb{C}$-Fuchsian or $\bR$-Fuchsian. The following is the B.Xie's result.
\begin{theorem}
Let $G \subset \textbf{SU}(2,1)$ be a non-elementary complex hyperbolic Kleinian group containing hyperbolic elements. Then $G$ is Fuchsian if and only if every element of $G$ has real trace.
\end{theorem}
Here, $\textbf{SU}(2,1)$ is the group of holomorphic isometries of complex hyperbolic $2$-space.\\
In this notes, we show that similar result holds in quaternionic hyperbolic case. The main theorem is the following.
\begin{theorem}
Let $G \subset \textbf{Sp}(2,1)$ be a non-elementary quaternionic hyperbolic Kleinian group in which every element of $G$ has a real trace. Then $G$ is Fuchsian.
\end{theorem}
The rest of this paper is organized as follows. In $\S$2, we give some necessary preliminaries on quaternionic hyperbolic space and in $\S$3, we prove the main theorem.

\section{Preliminaries}
For real hyperbolic case, we recommend Maskit's book \cite{Ma}, and for the complex hyperbolic case, B. Xie's paper \cite{Xie} is enough and for more details, we recommend \cite{Go}.
\subsection{Quaternionic hyperbolic space}
Let $\H^{2,1}$ be a quaternionic vector space of dimension $3$ with a
Hermitian form of signature $(2,1)$. An element of $\H^{2,1}$ is a
column vector $p=(p_1,p_2,p_3)^t$. Throughout this paper, we choose
the second Hermitian form on $\H^{2,1}$ given by the matrix $J$
$$
J=\left[\begin{matrix} 0 & 0 & 1 \\ 0 & 1 & 0 \\ 1 & 0 & 0
\end{matrix}\right].
$$
Thus $\<p,q\>=q^*Jp=\overline{q}^tJp=\overline{q_1}p_3+\overline{q_2}p_2+\overline{q_3}p_1$,
where $p=(p_1,p_2,p_3)^t, q=(q_1,q_2,q_3)^t \in \H^{2,1}$.\\
One model of a quaternionic hyperbolic space $\hh{2}$, which matches
the second Hermitian form and we will use throughout this paper,
is the \new{Siegel domain $\mathfrak{S}$}. It is defined by
identifying points of $\mathfrak{S}$ with their horospherical
coordinates, $p=(\zeta,v,u) \in \H \times Im(\H) \times \R_+$. The boundary of $\mathfrak{S}$ is given by $(\H \times Im(\H)) \cup \{\infty\}$, where $\infty$ is a distinguished point at infinity.\\
Define a map $\psi:\overline{\mathfrak{S}}\rightarrow \P \H^{2,1}$
by
$$
\psi:(\zeta,v,u)\mapsto \left[\begin{matrix}
-|\zeta|^2-u+v
\\ \sqrt{2}\zeta \\ 1\end{matrix}\right] \hbox{ for }
(\zeta,v,u)\in\overline{\mathfrak{S}}-\{\infty\} \hbox{   ; }
\psi:\infty\mapsto\left[\begin{matrix} 1
\\ 0 \\ 0\end{matrix}\right].
$$
Then $\psi$ maps $\mathfrak{S}$ homeomorphically to the set of
points $p$ in $\P \H^{2,1}$ with $\langle p,p\rangle<0$, and maps
$\partial \mathfrak{S}$ homeomorphically to the set of points $p$
in $\P \H^{2,1}$ with $\langle p,p\rangle=0$. We write $\psi(\tilde{p})=p$.\\
There is a metric on $\mathfrak{S}$ is called the Bergman metric and the isometry group of $\hh{2}$ with respect to this metric is
\begin{align*}
\textbf{PSp}(2,1) &=\{[A]:A \in GL(3,\H), \langle p,p' \rangle = \langle Ap,Ap' \rangle, p,p' \in \H^{2,1}\}\\
&= \{[A]:A \in GL(3,\H), J=A^*JA\},
\end{align*}
where $[A]:\P\H^2 \rightarrow \P\H^2; x\H \mapsto (Ax)\H$ for $A \in \textbf{Sp}(2,1)$. As in \cite{Kim}, we adopt the convention that the action of $\textbf{Sp}(2,1)$ on $\hh{2}$ is left and the action of projectivization of $\textbf{Sp}(2,1)$ is right action. If we write
$$
A=\left[\begin{matrix} a & b & c \\ d & e & f \\ g & h & l
\end{matrix}\right] \in \textbf{PSp}(2,1),
$$
$A^{-1}$ is written as
$$
A^{-1}=\left[\begin{matrix} \overline{l} & \overline{f} & \overline{c} \\ \overline{h} & \overline{e} & \overline{b} \\ \overline{g} & \overline{d} & \overline{a}
\end{matrix}\right] \in \textbf{PSp}(2,1).
$$
Then, from $AA^{-1}=A^{-1}A=I$, we get the following identities.
$$
a\overline{l}+b\overline{h}+c\overline{g}=1, a\overline{f}+b\overline{e}+c\overline{e}=0, a\overline{c}+|b|^2+c\overline{a}=0,
$$
$$
d\overline{l}+e\overline{h}+f\overline{g}=0, d\overline{f}+|e|^2+f\overline{d}=1, d\overline{c}+e\overline{b}+f\overline{a}=0,
$$
$$
g\overline{l}+|h|^2+l\overline{g}=0, g\overline{f}+h\overline{e}+l\overline{d}=0, g\overline{c}+h\overline{b}+l\overline{a}=1,
$$
$$
\overline{l}a+\overline{f}d+\overline{c}g=1, \overline{l}b+\overline{f}e+\overline{c}h=0, \overline{l}c+|f|^2+\overline{c}l=0,
$$
$$
\overline{h}a+\overline{e}d+\overline{b}g=0, \overline{h}b+|e|^2+\overline{b}h=1, \overline{h}c+\overline{e}f+\overline{b}l=0,
$$
$$
\overline{g}a+|d|^2+\overline{a}g=0, \overline{g}b+\overline{d}e+\overline{a}h=0, \overline{g}c+\overline{d}f+\overline{a}l=1.
$$
Similar to the complex hyperbolic space, totally geodesic submanifolds of quaternionic hyperbolic space are quaternionic line, complex line, $\H_{\C}^2$, and $\H_{\R}^2$. Then we say a Kleinian group $G$ in the group of holomorphic isometries of quaternionic hyperbolic space is Fuchsian if $G$ leaves invariant one of the above submanifolds.\\
The following propositions are needed in the proof of the main theorem.
\begin{proposition}
For two nonzero quaternions $a$ and $b$, if $ab \in \bR$ and $b$ is purely imaginary, then $a$ is also purely imaginary.
\end{proposition}

\begin{proof}
Let $a=a_0+a_1i+a_2j+a_3k$ and $b=b_1i+b_2j+b_3k$, where $a_t, b_t \in \bR$ for $t=0,1,2,3$. Then our claim is that $a_0=0$. Suppose not. Since $ab=-(a_1b_1+a_2b_2+a_3b_3)+(a_0b_1+a_2b_3-a_3b_2)i+(a_0b_2+a_3b_1-a_1b_3)j+(a_0b_3+a_1b_2-a_2b_1)k$ is real,
$$
a_0b_1+a_2b_3-a_3b_2=0, a_0b_2+a_3b_1-a_1b_3=0, a_0b_3+a_1b_2-a_2b_1=0.
$$
Since $a_0\neq0$, from the third equation, $\displaystyle{b_3=\frac{a_2b_1-a_1b_2}{a_0}}$. Substituting it to the first equation, we get $\displaystyle{b_1=\frac{a_1a_2+a_0a_3}{a_0^2+a_2^2}b_2}$ and again substituting $b_3$ and $b_1$ to the second equation, we get $(a_0^2+a_1^2+a_2^2+a_3^2)b_2=0$. Since $a$ is nonzero, $b_2$ must be zero. However if $b_2=0$, $b_1=b_3=0$, that is $b=0$ and it contracts that $b$ is nonzero.
\end{proof}

\begin{proposition}
For two nonzero quaternions $a$ and $b$, if $ab$ and $a\overline{b}$ are all real, then either $a$ and $b$ are real or $a$ and $b$ are purely imaginary.
\end{proposition}

\begin{proof}
Let $a=a_0+a_1i+a_2j+a_3k$ and $b=b_0+b_1i+b_2j+b_3k$, where $a_t, b_t \in \bR$ for $t=0,1,2,3$. Then the $i$-part of $ab$ and $a\overline{b}$ are $a_0b_1+a_1b_0+a_2b_3-a_3b_2$ and $-a_0b_1+a_1b_0-a_2b_3+a_3b_2$ respectively and since they must be zero, $a_1b_0=0$. By repeating for the $j$-part and $k$-part, we get $a_2b_0=a_3b_0=0$. \\
If $b_0=0$, then $b$ is purely imaginary and by above proposition, $a$ is also purely imaginary.\\
If $b_0 \neq 0$, then $a_1=a_2=a_3=0$, that is $a$ is real. Then, since $ab \in \bR$, $b$ is also real.
\end{proof}
\subsection{Quaternionic Cartan angular invariant}
The Cartan angular invariant is well-known invariant in complex hyperbolic geometry, but in quaternionic hyperbolic geometry, B.N.Apanasov and I.Kim first defined it in \cite{AK}. Here we give the definition and some properties which will be used in the proof of the main theorem.\\
The quaternionic Cartan angular invariant $\mathbb{A}_{\H}(p)$ of a triple $p=(p_1,p_2,p_3) \in (\hh{n} \cup \partial\hh{n})^3$ is the angle between the first coordinate line $\bR e_1=(\bR,0,0,0) \subset \bR^4 \cong \H$ and the radius vector of the quaternion equal to the Hermitian triple product $\displaystyle{\langle p_1,p_2,p_3 \rangle=\langle p_1,p_2 \rangle \langle p_2,p_3 \rangle \langle p_3,p_1 \rangle \in \H}$. Then $\mathbb{A}_{\H}(p)$ is independent of the choice of the lifts and $0 \leq \mathbb{A}_{\H}(p) \leq \pi/2$. Furthermore, $\mathbb{A}_{\H}(p)$ is invariant under permutations of the points $p_i$. An important fact on quaternionic Cartan angular invariant worth taking note of for our purposes is the following.
\begin{proposition}(Theorem 3.5 and 3.6 in \cite{AK})
A triple $p=(p_1,p_2,p_3) \in (\partial\hh{n})^3$ lies in the boundary of an $\H$-line if and only if $\mathbb{A}_{\H}(p)=\pi/2$, and lies in the same $\R$-circle if and only if $\mathbb{A}_{\H}(p)=0$.
\end{proposition}
\section{Proof of the main Theorem}
We may assume that a loxodromic element $A$ fixes $0$ and $\infty$, and $B$ is an arbitrary loxodromic element in $G$. In terms of matrices,
we can write $A=\left[\begin{matrix} \lambda\mu & 0 & 0 \\ 0 & \nu & 0 \\ 0 & 0 & \mu/\lambda \end{matrix} \right]$, and
$B=\left[\begin{matrix} a & b & c \\ d & e & f \\ g & h & l \end{matrix} \right]$, where $\mu,\nu \in \textbf{Sp}(1)$ and $\lambda >1$.(See \cite{Kim} or \cite{KP}) We claim that $\mu$ and $\nu$ are real number, so one of $\pm1$ for $\mu,\nu \in \textbf{Sp}(1)$.
Since $trA=\lambda\mu+\nu+\mu/\lambda \in \bR$, writing $\mu=\mu_1+\mu_2i+\mu_3j+\mu_4k$ and $\nu=\nu_1+\nu_2i+\nu_3j+\nu_4k$ for $\mu_t,\nu_t \in \bR$ where $t=1,2,3,4$, $trA=(\lambda+1/\lambda)(\mu_1+\mu_2i+\mu_3j+\mu_4k)+(\nu_1+\nu_2i+\nu_3j+\nu_4k) \in \bR$.
Hence $(\lambda+1/\lambda)\mu_t+\nu_t=0$ for $t=2,3,4$.\\
Furthermore,
\begin{align*}
tr(A^2) & =(\lambda^2+1/\lambda^2)\mu^2+\nu^2\\ &= (\lambda^2+1/\lambda^2)(\mu_1^2-\mu_2^2-\mu_3^2-\mu_4^2+2\mu_1\mu_2i+2\mu_1\mu_3j+2\mu_1\mu_4k)\\
&+(\nu_1^2-\nu_2^2-\nu_3^2-\nu_4^2+2\nu_1\nu_2i+2\nu_1\nu_3j+2\nu_1\nu_4k) \in \bR.
\end{align*}
Hence, $(\lambda^2+1/\lambda^2)\mu_1\mu_t+\nu_1\nu_t=\mu_t[(\lambda^2+1/\lambda^2)\mu_1-(\lambda+1/\lambda)\nu_1]=0$ for $t=2,3,4$.\\
If $\mu_2=\mu_3=\mu_4=0$, then, by above identity, $\nu_2=\nu_3=\nu_4=0$, so $\mu$ and $\nu$ are real number and $\mu,\nu \in \{1,-1\}$. If there is at least one nonzero number among $\{\mu_2,\mu_3,\mu_4\}$, $(\lambda^2+1/\lambda^2)\mu_1-(\lambda+1/\lambda)\nu_1=0$, so we can write $\displaystyle{\nu=\frac{\lambda^4+1}{\lambda(\lambda^2+1)}\mu_1-(\lambda+\frac{1}{\lambda})(\mu_2i+\mu_3j+\mu_4k)}$.\\
Now let's consider $A^4$. Then,
\begin{align*}
tr(A^4) & =(\lambda^4+1/\lambda^4)\mu^4+\nu^4\\ &= (\lambda^4+1/\lambda^4)(\mu_1+\mu_2i+\mu_3j+\mu_4k)^4\\ &+[\frac{\lambda^4+1}{\lambda(\lambda^2+1)}\mu_1-(\lambda+\frac{1}{\lambda})
(\mu_2i+\mu_3j+\mu_4k)]^4 \in \bR.
\end{align*}
By calculation, the $i$-part is
\begin{align*}
& 4\mu_1\mu_2(\lambda^4+\frac{1}{\lambda^4})(\mu_1^2-\mu_2^2-\mu_3^2-\mu_4^2)\\
& -4\frac{\lambda^4+1}{\lambda(\lambda^2+1)}\mu_1(\lambda+\frac{1}{\lambda})\mu_2\
[(\frac{\lambda^4+1}{\lambda(\lambda^2+1)}\mu_1)^2-(\lambda+\frac{1}{\lambda})^2(\mu_2^2+\mu_3^2+\mu_4^2)]\\
&= \frac{4\mu_1\mu_2}{\lambda^4}[(\lambda^8+1)(\mu_1^2-\mu_2^2-\mu_3^2-\mu_4^2)-\frac{(\lambda^4+1)^3}{(\lambda^2+1)^2}\mu_1^2+
(\lambda^4+1)(\lambda^2+1)^2(\mu_2^2+\mu_3^2+\mu_4^2)]\\
&= \frac{4\mu_1\mu_2}{\lambda^4}[\frac{2\lambda^2(\lambda^2-1)^2(\lambda^4+\lambda^2+1)}{(\lambda^2+1)^2}\mu_1^2
+2\lambda^2(\lambda^4+\lambda^2+1)(\mu_2^2+\mu_3^2+\mu_4^2)]\\
&= \frac{8(\lambda^4+\lambda^2+1)}{\lambda^2}[\frac{(\lambda^2-1)^2}{(\lambda^2+1)^2}\mu_1^2+\mu_2^2+\mu_3^2+\mu_4^2]\mu_1\mu_2,
\end{align*}
and which must be zero. Since $\lambda >1$, $\mu_1\mu_2=0$. By repeating the same argument for $j$-part and $k$-part, we get $\mu_1\mu_3=\mu_1\mu_4=0$. Since there is a nonzero number among $\{\mu_2,\mu_3,\mu_4\}$, $\mu_1=0$ and hence $\nu_1=0$. That is, $\mu$ and $\nu$ are purely imaginary. Then we can write $\mu=\mu_2i+\mu_3j+\mu_4k$ and $\nu=-(\lambda+\frac{1}{\lambda})(\mu_2i+\mu_3j+\mu_4k)$. It contradicts that $\mu,\nu \in \textbf{Sp}(1)$ for $\lambda >1$. Therefore $\mu,\nu=\pm1$ and we proved the claim. Hence $A=\left[\begin{matrix} \lambda & 0 & 0 \\ 0 & 1 & 0 \\ 0 & 0 & \frac{1}{\lambda} \end{matrix} \right]$ or $A=\left[\begin{matrix} \lambda & 0 & 0 \\ 0 & -1 & 0 \\ 0 & 0 & \frac{1}{\lambda} \end{matrix} \right]$.\\
Now let's assume that $A=\left[\begin{matrix} \lambda & 0 & 0 \\ 0 & 1 & 0 \\ 0 & 0 & \frac{1}{\lambda} \end{matrix} \right]$. In the case of $A=\left[\begin{matrix} \lambda & 0 & 0 \\ 0 & -1 & 0 \\ 0 & 0 & \frac{1}{\lambda} \end{matrix} \right]$, the proof is almost the same as in this case.\\
Since every element of $G$ has real trace,
$$\displaystyle{tr(B)=a+e+l, tr(AB)=\lambda a+e+l/\lambda, tr(A^{-1}B)=a/\lambda+e+\lambda l}$$ are all real. Solving for $a,e,$ and $l$, we get $a,e,l \in \bR$. This shows that every element of $G$ has real diagonal elements. Now let's consider the following two matrices of $B^2$ and $B'=BAB^{-1}$.

$$
B^2=\left[\begin{matrix} a^2+bd+cg & * & * \\ * & db+e^2+fh & * \\ * & * & gc+hf+l^2 \end{matrix} \right],
$$
$$
B'=BAB^{-1}=
\left[\begin{matrix} \lambda a\overline{l}+b\overline{h}+c\overline{g}/\lambda & \lambda a\overline{f}+b\overline{e}+c\overline{d}/\lambda & \lambda a\overline{c}+b\overline{b}+c\overline{a}/\lambda \\ \lambda d\overline{l}+e\overline{h}+f\overline{g}/\lambda & \lambda d\overline{f}+e\overline{e}+f\overline{d}/\lambda & \lambda d\overline{c}+e\overline{b}+f\overline{a}/\lambda \\ \lambda g\overline{l}+h\overline{h}+l\overline{g}/\lambda &\lambda g\overline{f}+h\overline{e}+l\overline{d}/\lambda & \lambda g\overline{c}+h\overline{b}+l\overline{a}/\lambda \end{matrix} \right].
$$
Since diagonal elements of $B^2$ are real and $a,e,l$ are real, we get $bd+cg, db+fh, gc+hf \in \bR$. Similarly, from $B'$, $b\overline{h}+c\overline{g}/\lambda, \lambda d\overline{f}+f\overline{d}/\lambda, \lambda g\overline{c}+h\overline{b} \in \bR$. By the identity $a\overline{l}+b\overline{h}+c\overline{g}=1$, we know that $b\overline{h}+c\overline{g}$ is real and combining it with $b\overline{h}+c\overline{g}/\lambda$ is real, we get $b\overline{h}$ and $c\overline{g}$ are real. Also, from $d\overline{f}+|e|^2+f\overline{d}=1$, $d\overline{f}+f\overline{d}$ is real and combining it with $\lambda d\overline{f}+f\overline{d}/\lambda$ is real, $d\overline{f}$ is also real.\\
By the way, the $(1,1)$-entry of $B^2AB^{-1}$ is
\begin{align*}
& \lambda a^2\overline{l}+ab\overline{h}+\frac{1}{\lambda}ac\overline{g}+\lambda bd\overline{l}+be\overline{h}+\frac{1}{\lambda}bf\overline{g}+\lambda cg\overline{l}+c|h|^2+\frac{1}{\lambda}cl\overline{g}\\ &= (\lambda a^2\overline{l}+ab\overline{h}+\frac{1}{\lambda}ac\overline{g}+eb\overline{h}+\frac{1}{\lambda}c\overline{g}l)
+\lambda(bd+cg)\overline{l}+\frac{1}{\lambda}bf\overline{g}+c|h|^2 \in \bR.
\end{align*}
Since $a,e,l,b\overline{h},c\overline{g}$, and $bd+cg$ are real, $\displaystyle{\frac{1}{\lambda}bf\overline{g}+c|h|^2 \in \bR}$.\\
From the identities $g\overline{l}+|h|^2+l\overline{g}=0$ and $d\overline{l}+e\overline{h}+f\overline{g}=0$,
\begin{align*}
c|h|^2+\frac{1}{\lambda}bf\overline{g} &=c(-g\overline{l}-l\overline{g})+\frac{1}{\lambda}b(-d\overline{l}-e\overline{h})\\
&= -(cg+\frac{1}{\lambda}bd)\overline{l}-lc\overline{g}-\frac{1}{\lambda}eb\overline{h}.
\end{align*}
In the last equation, we use the fact that $e$ and $l$ are real. Since $b\overline{h}$ and $c\overline{g}$ are real,
we get $\displaystyle{(cg+\frac{1}{\lambda}bd)l \in \bR}$.\\
If $l\neq0$, $cg+\frac{1}{\lambda}bd \in \bR$, and since $cg+bd \in \bR$, we obtain $cg,bd \in \bR$. Then, by Proposition 2.2, either $c$ and $g$ are real or $c$ and $g$ are purely imaginary.\\
If $l=0$, from the identities $\overline{l}c+|f|^2+\overline{c}l=0$ and $g\overline{l}+|h|^2+l\overline{g}=0$, $f=h=0$. Looking into the (3,3)-entry of $B^2$ and $B'$, $gc$ and $g\overline{c}$ are real. Then, by Proposition 2.2, either $c$ and $g$ are real or $c$ and $g$ are purely imaginary.\\
\textbf{Case I}: $c$ and $g$ are purely imaginary.\\
From the identity $\overline{g}a+|d|^2+\overline{a}g=0$, $d=0$ because $a$ is real and $g$ is purely imaginary. Similarly, from identities $\overline{a}c+|b|^2+\overline{c}a=0, \overline{l}c+|f|^2+\overline{c}l=0$, and $g\overline{l}+|h|^2+l\overline{g}=0$, we get $b=f=h=0$. Hence, $B=\left[\begin{matrix} a & 0 & c \\ 0 & e & 0 \\ g & 0 & l \end{matrix} \right]$, where $a,e,l$ are real and $c,g$ are purely imaginary. Then $A$ and $B$ leave invariant a quaternionic line $H$ of polar vector $\left[\begin{matrix} 0 \\ 1 \\ 0 \end{matrix} \right]$.\\
Now let $B^{*}=\left[\begin{matrix} a' & b' & c' \\ d' & e' & f' \\ g' & h' & l' \end{matrix} \right]$ be any other element of $G$. Then, $a',e',l'$ are real and diagonal elements of $B^{*}B$, which are $a'a+c'g, e'e$, and $g'c+l'l$, are also real. Hence $c'g$ and $g'c$ are real. Then, by Proposition 2.1, $c'$ and $g'$ are also purely imaginary, so $b'=d'=f'=h'=0$ as above. Therefore, we conclude that $G$ leaves invariant a quaternionic line $H$.\\
\textbf{Case II}: $c$ and $g$ are real.\\
Calculating the quaternionic Cartan angular invariant of three points $0,\infty,B(0)$, since $\langle 0,\infty,B(0) \rangle=\overline{l}c \in \bR$, $\textit{A}_{\mathbb{H}}(0,\infty,B(0))=0$. Hence, by Proposition 2.3, these three points lie in an $\bR$-circle, so we may normalize so that this $\bR$-circle is $\mathbb{H}^2_{\bR}$. Then $B(0)=\left[\begin{matrix} c \\ f \\ l \end{matrix} \right] \in \mathbb{H}^2_{\bR}$, that is, $f$ is real. We notice that $f$ cannot be zero, because if $f=0$, from the identity $\overline{l}c+|f|^2+\overline{c}l=0$, $cl$ must be zero, which means that $B(0)$ is either $0$ or $\infty$, and it contracts that $G$ is non-elementary. Since $f$ is nonzero and real, $d$ is also real for $d\overline{f} \in \bR$. Then, similarly, $d$ cannot be zero by the identity $\overline{g}a+|d|^2+\overline{a}g=0$. Since $bd+cg \in \bR$ and $gc+hf \in \bR$, $b$ and $h$ are also real. Therefore, we conclude that $B=\left[\begin{matrix} a & b & c \\ d & e & f \\ g & h & l \end{matrix} \right] \in \textbf{SO}(2,1)$, hence $\langle a,b \rangle \subset \textbf{SO}(2,1)$.\\
Now let $B_{*}=\left[\begin{matrix} a' & b' & c' \\ d' & e' & f' \\ g' & h' & l' \end{matrix} \right]$ be another element of $G$. Then, diagonal elements $a',e'$, and $l'$ are real. Also, diagonal elements of $BB_{*}$ to be real, $bd'+cg', db'+fh', gc'+hf' \in \bR$. Similarly, if we consider diagonal elements of $BAB_*=\left[\begin{matrix} \lambda aa'+bd'+\frac{1}{\lambda}cg' & * & * \\ * & \lambda db'+ee'+\frac{1}{\lambda}fh' & * \\ * & * & \lambda gc'+hf'+\frac{1}{\lambda}ll' \end{matrix} \right]$, we know that $\displaystyle{bd'+\frac{1}{\lambda}cg',\lambda db'+\frac{1}{\lambda}fh',\lambda gc'+hf' \in \bR}$. Since $\lambda>1$, $bd',cg',db',fh',gc',hf' \in \bR$. Hence $b'$ and $h'$ are real because $f$ and $d$ are nonzero.\\
Furthermore, since $d$ and $f$ are nonzero, from identities $\overline{l}c+|f|^2+\overline{c}l=0$, and  $\overline{g}a+|d|^2+\overline{a}g=0$, $c$ and $g$ cannot be zero. Hence, $c'$ and $g'$ are real because $cg'$ and $gc'$ are real. Again, looking into diagonal elements of $B^{-1}B_*$, we can find that $fd'$ and $df'$ are real. Since $d$ and $f$ are nonzero real, $d'$ and $f'$ are real.\\
Therefore, $B_*$ is also in $\textbf{SO}(2,1)$ and it shows that every element of $G$ preserves $\mathbb{H}^2_{\bR}$.



\begin{thebibliography}{99}
   \bibitem{AK} B. N. Apanasov and I. Kim, Cartan angular invariant and deformations of rank 1 symmetric spaces,
    Sbornik: Mathematics 198:2 (2007) 147-169.
    \bibitem{Go} W. M. Goldman, Complex hyperbolic Geometry, Oxford Univ. Press, (1999).
    \bibitem{Kim} D. Kim, Discreteness Criterions of Isometric Subgroups for Quaternionic Hyperbolic Space,
    Geometriae Dedicata 106 (2004), 51-78.
    \bibitem{KP} I. Kim and J. R. Parker, Geometry of quaternionic hyperbolic manifolds,
    Math. Proc. Camb. Phil. Soc. 135 (2003), 291-320.
    \bibitem{Ma} B. Maskit, Kleinian groups, Springer-Verlag, (1988).
    \bibitem{Xie} B. Xie, The complex Kleinian groups with an invariant totally geodesic submanifold, preprint(arXiv:1006.5581).
\end{thebibliography}
\end{document}